\documentclass[twoside,a4paper,11pt]{article}
\usepackage{amsthm,amsmath,amssymb,amscd,graphicx,amsfonts,stmaryrd}
\usepackage{euscript}
\usepackage{mathrsfs}
\usepackage[all]{xy}
\usepackage[margin=25mm]{geometry}
\usepackage{bigints}
\usepackage{mathtools}
\usepackage{enumitem}
\usepackage{comment}
\usepackage{fancyhdr}
\usepackage{latexsym}
\usepackage{amsfonts}
\usepackage[normalem]{ulem}
\usepackage{bm}
\usepackage[d]{esvect}
\usepackage{array}
\usepackage{tikz}
\usetikzlibrary{positioning, intersections, calc, arrows.meta,math}

\usepackage{hyperref}

\allowdisplaybreaks[1]
\thepage

\mathtoolsset{showonlyrefs=true}

\numberwithin{equation}{section}

\theoremstyle{definition}
\newtheorem{exm} {Example}[section]
\newtheorem{dfn}[exm] {Definition}
\newtheorem{rem}[exm] {Remark}

\theoremstyle{plane}
\newtheorem{lem}[exm]{Lemma}
\newtheorem{prop}[exm]{Proposition}
\newtheorem{thm}[exm]{Theorem}
\newtheorem{cor}[exm]{Corollary}

\newcommand{\filt}[0]{\EuScript{F}}
\newcommand{\domain}[0]{\mathcal{F}}
\newcommand{\form}[0]{\mathcal{E}}
\newcommand{\resolop}[0]{R}
\newcommand{\resol}[0]{r}
\newcommand{\sigalg}[0]{\EuScript{M}}
\newcommand{\smooth}[0]{\mathcal{S}}

\title{Generalized Kac's moment formula for positive continuous additive functionals of Markov processes} 
\date{}
\author{Naotaka Kajino\thanks{Research Institute for Mathematical Sciences, Kyoto University, nkajino@kurims.kyoto-u.ac.jp} 
\and Ryoichiro Noda\thanks{Research Institute for Mathematical Sciences, Kyoto University, sgrndr@kurims.kyoto-u.ac.jp}}

\begin{document}

\maketitle
\begin{abstract}
  We establish a formula for moments of certain random variables involving positive continuous
  additive functionals (PCAFs) of standard processes which have absolutely continuous transition
  functions and are in duality with standard processes with absolutely continuous transition functions,
  generalizing the classical Kac's moment formula. In particular, all our results are
  applicable to the more familiar case of symmetric Hunt processes which are associated with
  regular Dirichlet forms and have absolutely continuous transition functions.

  \vskip.2cm

		
	\vskip.2cm
		
\end{abstract}


\section{Introduction} \label{sec: Intro}

The aim of this paper is to extend Kac's moment formula \cite{Kac_51_connection} (see \eqref{1. eq: original Kac's formula} below)
to general positive continuous additive functionals (PCAFs) of Markov processes
and to certain more general random variables involving PCAFs.
Let us first recall Kac's moment formula for the Brownian motion.
Let $X = (X_{t})_{t \geq 0}$ be the one-dimensional Brownian motion,
$f \colon \mathbb{R} \to [0, \infty)$ bounded and Borel measurable,
and $(p_{t}(x,y))_{t>0, \, x, y \in \mathbb{R}}$ the transition density of the Brownian motion, i.e.,
$p_{t}(x,y) = (2\pi)^{-1/2} \exp(-(x-y)^{2}/2t)$.
Consider the additive functional $A = (A_{t})_{t \geq 0}$ of $X$
given by setting $A_{t} \coloneqq \int_{0}^{t} f(X_{s})\, ds$.
Then, Kac's moment formula reads as follows: 
for any $t \in (0, \infty)$, $x \in \mathbb{R}$, and positive integer $k$,
\begin{align}
  E_{x}[A_{t}^{k}]
  &= 
  k! 
  \int_{0}^{t} dt_{1} \int_{t_{1}}^{t} dt_{2} \cdots \int_{t_{k-1}}^{t} dt_{k} 
  \int_{\mathbb{R}} f(y_{1})\, dy_{1} \int_{\mathbb{R}} f(y_{2})\, dy_{2}\\
  &\qquad 
  \cdots \int_{\mathbb{R}} f(y_{k})\, dy_{k}\,
  p_{t_{1}}(x, y_{1}) p_{t_{2}-t_{1}}(y_{1}, y_{2}) \cdots p_{t_{k}-t_{k-1}}(y_{k-1}, y_{k}).
  \label{1. eq: original Kac's formula}
\end{align}
Kac's moment formula plays fundamental roles in studying occupation times and Feynman--Kac semigroups
(see, e.g., \cite{Darling_Kac_57_occupation,Sznitman_98_Brownian}).
To obtain \eqref{1. eq: original Kac's formula}, 
one does not need to use any specific properties of Brownian motion, except for the Markov property. 
Therefore, it is not difficult to extend \eqref{1. eq: original Kac's formula} to general Markov processes with transition densities. 
On the other hand, the condition that the additive functional $A$ is of the form $A_{t}=\int_{0}^{t} f(X_{s}) ds$ is crucial for the applicability of Fubini's theorem, 
and it is not clear how \eqref{1. eq: original Kac's formula} can be extended to more general additive functionals.

In this paper, we first state and prove our results in the framework of a symmetric Hunt process $X$
whose Dirichlet form is regular and whose transition function is absolutely continuous,
and for positive continuous additive functionals (PCAFs) in the strict sense of $X$.
(The precise definitions of these objects are recalled in Section~\ref{sec: Setting and main results} below.)
In fact, the choice of this setting is solely for the sake of the better familiarity to the reader,
and our results hold, without any changes in the proofs, in the more general framework,
as detailed in Section \ref{sec: extension to processes with duals} below, of a standard process
with absolutely continuous transition function and in duality with another such standard process.

Since some preparation is required to state the main theorems of this paper (Theorems \ref{2. thm: main result} and \ref{2. thm: generalized Kac's formula}),
here we explain briefly what our version of Kac's moment formula \eqref{1. eq: original Kac's formula}
looks like in the present general setting.
Fix a locally compact separable metrizable topological space $S$, a Radon measure $m$ on $S$ with full support,
and an $m$-symmetric Hunt process $X = (X_{t})_{t \geq 0}$ on $S$ whose Dirichlet form $(\form, \domain)$ is regular,
and assume that $X$ admits a jointly Borel measurable transition density $(p_{t}(x,y))_{t > 0, \, x, y \in S}$ with respect to $m$.
(As mentioned in the previous paragraph, our results hold as long as
$m$ is a $\sigma$-finite Borel measure on $S$ and $X$ is a standard process
on $S$ with a jointly Borel measurable transition density $(p_{t}(x,y))_{t > 0, \, x, y \in S}$
with respect to $m$ such that $(t,x,y)\mapsto p_{t}(y,x)$ is a transition density of another
standard process on $S$ with respect to $m$.) For each $\alpha \in (0, \infty)$,
we define the $\alpha$-potential density $(\resol_{\alpha}(x,y))_{x,y \in S}$ of $X$ by setting 
$\resol_{\alpha}(x,y) \coloneqq \int_{0}^{\infty} e^{-\alpha t} p_{t}(x,y)\, dt$.
Let $A=(A_{t})_{t \geq 0}$ be a PCAF in the strict sense of $X$.
By the Revuz correspondence (see Lemma \ref{2. lem: Revuz correspondence in the strict sense}),
there exists a unique $\sigma$-finite Borel measure $\mu$ on $S$, called the \emph{Revuz measure} of $A$, satisfying
\begin{equation}  \label{1. eq: Revuz correspondence}
  E_{x}\biggl[
    \int_{0}^{\infty} e^{-\alpha t} f(X_{s})\, dA_{s}
  \biggr]
  = 
  \int_{S} \resol_{\alpha}(x,y) f(y)\, \mu(dy)
\end{equation}
for any $\alpha \in (0, \infty)$, $x \in S$, and Borel measurable function $f \colon S \to [0, \infty]$.
In this setting, our second main theorem (Theorem \ref{2. thm: generalized Kac's formula}) implies that,
for any $t \in (0, \infty)$, $x \in S$, and positive integer $k$,
\begin{align} 
  E_{x}[A_{t}^{k}]
  &= 
  k! 
  \int_{0}^{t} dt_{1} \int_{t_{1}}^{t} dt_{2} \cdots \int_{t_{k-1}}^{t} dt_{k} 
  \int_{S} \mu(dy_{1}) \int_{S} \mu(dy_{2}) \cdots \int_{S} \mu(dy_{k})\\
  &\qquad
  p_{t_{1}}(x, y_{1}) p_{t_{2}-t_{1}}(y_{1}, y_{2}) \cdots p_{t_{k}-t_{k-1}}(y_{k-1}, y_{k})
  \label{1. eq: generalized moment formula}
\end{align}
(see Corollary \ref{2. cor: applications} below).
This is a generalization of \eqref{1. eq: original Kac's formula}.
Indeed, if we consider a PCAF $A=(A_{t})_{t \geq 0}$ of $X$ given by $A_{t} \coloneqq \int_{0}^{t} f(X_{s})\, ds$
for a bounded Borel measurable function $f \colon S \to [0, \infty)$,
then the measure $\mu$ satisfying \eqref{1. eq: Revuz correspondence} is $\mu(dx) = f(x)\, m(dx)$.
Thus, \eqref{1. eq: original Kac's formula} is recovered from \eqref{1. eq: generalized moment formula}.

The special case of \eqref{1. eq: generalized moment formula}
where $X$ is the $\alpha$-dimensional Bessel process on $[0, \infty)$ for $\alpha\in(0,2)$
and $A$ is its local time at $0$ was obtained by Molchanov and Ostrovskii
in \cite[p.~129]{Molchanov_Ostrovskii_69_stable_trace},
by applying \eqref{1. eq: original Kac's formula} to absolutely continuous PCAFs
$A^{(n)}$ of $X$ converging to $A$ and then by letting $n\to\infty$
on the basis of explicit quantitative estimates on $X$. Unlike this argument in
\cite{Molchanov_Ostrovskii_69_stable_trace}, our proof of the main results of this paper including
\eqref{1. eq: generalized moment formula} is based only on the strong Markov property of $X$ and
a change-of-variable formula \cite[Chapter V, Lemma 2.2]{Blumenthal_Getoor_68_Markov}
for Lebesgue--Stieltjes integrals and does not require any quantitative estimates on $X$ or $A$.

There is also a work in a similar research direction by Fitzsimmons and Pitman \cite{Fitzsimmons_Pitman_99_Kac},
where they considered more general Markov processes 
and their moment formula is described in terms of certain potential operators associated with PCAFs
rather than transition densities and corresponding Revuz measures as in \eqref{1. eq: generalized moment formula}.
A striking difference between their results and ours is 
that they considered moments of additive functionals at Markov killing times of $X$,
whereas we consider moments of them at deterministic times.
Here, a Markov killing time of $X$ refers to a random time with the property that
the process $X$ killed at the time is again Markov,
and a finite deterministic time is usually not a Markov killing time of $X$.
For further discussions, 
see Remarks \ref{2. rem: comment on F-P results} and \ref{2. rem: extension to part process}.

We expect that our generalized moment formula will be useful in the study of various objects involving PCAFs
such as Feynman--Kac semigroups, 
just as the classical Kac's moment formula is. 
In fact, in \cite{Noda_pre_Continuity},
on the basis of our results,
the convergence of PCAFs is shown to be implied by the convergence of the potentials of the corresponding Revuz measures.

The remainder of this article is organized as follows.
In Section \ref{sec: Setting and main results},
we introduce the framework of an $m$-symmetric Hunt process whose Dirichlet form is regular and
whose transition function is absolutely continuous, state the main theorems of this article,
Theorems \ref{2. thm: main result} and \ref{2. thm: generalized Kac's formula}, and give some corollaries of them.
Then, in Section \ref{sec: Proofs of the main results}, we prove Theorems \ref{2. thm: main result} and \ref{2. thm: generalized Kac's formula}.
Finally, in Section~\ref{sec: extension to processes with duals}, 
we present the framework of a standard process with absolutely continuous transition function
and in duality with another such standard process, and give the background necessary for the
proofs of our results to extend immediately to this more general setting.


\section{Setting and main results} \label{sec: Setting and main results}

This section is divided into two subsections.
In Subsection \ref{sec: PCAFs and smooth measures},
we set out the framework for the discussions of Sections \ref{sec: Setting and main results} and \ref{sec: Proofs of the main results}
and introduce PCAFs and smooth measures in the strict sense.
Then, in Subsection \ref{sec: main results}, we present our main theorems, Theorems \ref{2. thm: main result} and \ref{2. thm: generalized Kac's formula},
and two corollaries of them (Corollaries \ref{2. cor: applications} and \ref{2. cor: finite exponential moment}).
Throughout this paper,
we fix a locally compact separable metrizable topological space $S$, and
write $S_{\Delta} = S \cup \{\Delta\}$ for the one-point compactification of $S$.
(N.B. If $S$ is compact, then we add $\Delta$ to $S$ as an isolated point.)
Any $[-\infty,\infty]$-valued function $f$ defined on $S$ is regarded as a function on $S_{\Delta}$ by setting $f(\Delta) \coloneqq 0$,
and we define $\|f\|_{\infty} \coloneqq \sup\{|f(x)| \mid x \in S\}$ for each such $f$.
The symbol $\mathbb{N}$ denotes the set of positive integers.
For $a,b \in [-\infty,\infty]$, or $[-\infty,\infty]$-valued functions $a,b$ on a common set,
we set $a\vee b \coloneqq \max\{a,b\}$ and $a\wedge b \coloneqq \min\{a,b\}$.
For a non-empty set $E$ and $E_{0} \subseteq E$,
we write $1_{E_{0}}=1^{E}_{E_{0}} \colon E \to \mathbb{R}$ for the function defined by setting
$1_{E_{0}}(x) \coloneqq 1$ for $x \in E_{0}$ and $1_{E_{0}}(x) \coloneqq 0$ for $x \in E \setminus E_{0}$.
Given a topological space $E$, we write $\mathcal{B}(E)$ for the Borel $\sigma$-algebra of $E$.


\subsection{Setting, PCAFs and smooth measures in the strict sense} \label{sec: PCAFs and smooth measures}

In this subsection,
we clarify the setting for our discussions
and introduce the main objects: PCAFs and smooth measures in the strict sense.
Our results are established within the theory of regular symmetric Dirichlet forms and symmetric Hunt processes.
For details of this theory, the reader is referred to \cite{Chen_Fukushima_12_Symmetric,Fukushima_Oshima_Takeda_11_Dirichlet}.

We first fix the setting that is assumed throughout Sections \ref{sec: Setting and main results} and \ref{sec: Proofs of the main results}.
We let $m$ be a Radon measure on $S$ with full support, $(\form, \domain)$ be a regular symmetric Dirichlet form on $L^{2}(S, m)$
(see \cite[Section 1.1]{Fukushima_Oshima_Takeda_11_Dirichlet} for the definition of the notion of regular symmetric Dirichlet form),
and $X = (\Omega, \sigalg, (X_{t})_{t \in [0, \infty]}, (P_{x})_{x \in S_{\Delta}}, (\theta_{t})_{t \in [0,\infty]})$
be an $m$-symmetric Hunt process on $S$ whose Dirichlet form is $(\form, \domain)$ in the sense of \cite[Theorem 7.2.1]{Fukushima_Oshima_Takeda_11_Dirichlet}.
Here, $\theta_{t}$ denotes the shift operator,
i.e., a map $\theta_{t} \colon \Omega \to \Omega$ satisfying $X_{s} \circ \theta_{t} = X_{s+t}$ for any $s \in [0, \infty]$.
We let $\zeta$ denote the life time of $X$, i.e., a $[0,\infty]$-valued function on $\Omega$
satisfying $\{ X_{t} = \Delta \} = \{ \zeta \leq t \}$ for each $t \in [0,\infty]$,
and write $\filt_{*} = (\filt_{t})_{t \in [0,\infty]}$ 
for the minimum augmented admissible filtration of $X$ in $\Omega$ (see \cite[p.~397]{Chen_Fukushima_12_Symmetric}).
A Borel subset $N$ of $S$ is said to be \emph{properly exceptional} for $X$ if it satisfies $m(N) = 0$
and $P_{x}(\{X_{t} \mid t \in (0, \infty)\} \subseteq S_{\Delta} \setminus N) = 1$ for all $x \in S \setminus N$.
(Note that by \cite[Theorem A.2.3]{Fukushima_Oshima_Takeda_11_Dirichlet},
which is applicable due to the present assumption that $X$ is a Hunt process on $S$,
for each Borel subset $N$ of $S$ and each $x\in S_{\Delta}$,
$P_{x}(\{X_{t} \mid t \in (0, \infty)\} \subseteq S_{\Delta} \setminus N) = 1$ if and only if
$P_{x}(\{X_{t}, X_{t-} \mid t \in (0, \infty)\} \subseteq S_{\Delta} \setminus N) = 1$, where
$X_{t-}(\omega) \coloneqq \lim_{s \uparrow t}X_{s}(\omega) \in S_{\Delta}$ for $(t,\omega) \in (0,\infty) \times \Omega$.)
We assume that $X$ satisfies the following \emph{absolute continuity condition} \ref{2. item: absolute continuity} (with respect to $m$).
\begin{enumerate} [label = \textup{(AC)}, leftmargin = *] 
  \item \label{2. item: absolute continuity}
    For all $x \in S$ and $t \in (0, \infty)$, the Borel measure $P_{x}(X_{t} \in dy)$ on $S$ is absolutely continuous with respect to $m(dy)$.
\end{enumerate}
Then, by \cite[Theorem~2]{Yan_88_A_formula},
there exists a unique Borel measurable function $p \colon (0, \infty) \times S \times S \to [0,\infty]$ satisfying, for any $s, t \in (0, \infty)$ and $x, y \in S$, 
$P_{x}(X_{t} \in dz) = p_{t}(x,z)\, m(dz)$ (as Borel measures on $S$), $p_{t}(x,y) = p_{t}(y,x)$, and $p_{t+s}(x, y) = \int_{S} p_{t}(x, z) p_{s}(z, y)\, m(dz)$.
The function $p$ is called the \emph{transition density} (or \emph{heat kernel}) of $X$ (with respect to $m$).
For convenience, we extend the domain of $p$ to $(0, \infty] \times S \times S_{\Delta}$ by setting 
\begin{equation} \label{2. eq: extended heat kernel}
  p_{t}(x, \Delta) 
  \coloneqq 
  1- \int_{S} p_{t}(x, y)\, m(dy)
  = 
  P_{x}(X_{t} = \Delta)
  \qquad \textrm{and} \qquad
  p_{\infty}(x, y) \coloneqq 1_{\{\Delta\}}(y)
\end{equation}
for $t \in (0, \infty)$, $x \in S$, and $y \in S_{\Delta}$,
and extend the measure $m$ to a Borel measure $m_{\Delta}$ on $S_{\Delta}$ 
by setting $m_{\Delta} \coloneqq m(\cdot \cap S) + \delta_{S_{\Delta},\Delta}$,
where $\delta_{S_{\Delta},\Delta}$ denotes the Dirac measure on $S_{\Delta}$ putting mass $1$ at $\Delta$.
It is then easy to check that, for any $t \in (0, \infty]$, $x \in S$, and Borel measurable function $f \colon S_{\Delta} \to [0, \infty]$
(that is not necessarily $0$ at $\Delta$),
\begin{equation} \label{2. eq: extended heat kernel property}
  E_{x}[f(X_{t})]
  = 
  \int_{S_{\Delta}} p_{t}(x,y) f(y)\, m_{\Delta}(dy).
\end{equation}
For each $\alpha \in (0, \infty)$,
the \emph{$\alpha$-potential density} $\resol_{\alpha} \colon S \times S \to [0, \infty]$ of $X$ is defined by setting
\begin{equation} \label{2. eq: potential density}
  \resol_{\alpha}(x,y)
  \coloneqq 
  \int_{0}^{\infty} e^{-\alpha t} p_{t}(x,y)\, dt,
\end{equation}
so that by Fubini's theorem we have, for any $\alpha,\beta \in (0,\infty)$ and $x,y \in S$,
\begin{equation} \label{2. eq: density resolvent equation}
\resol_{\alpha \wedge \beta}(x,y) = \resol_{\alpha \vee \beta}(x,y)
  + |\alpha - \beta| \int_{S} \resol_{\alpha}(x,z)\resol_{\beta}(z,y) \, m(dz).
\end{equation}
Given a Borel measure $\nu$ on $S$ and $\alpha \in (0, \infty)$,
we define $\resolop_{\alpha}\nu \colon S \to [0, \infty]$ by setting
\begin{equation} \label{2. eq: potential function}
  \resolop_{\alpha}\nu(x) 
  \coloneqq 
  \int_{S} \resol_{\alpha}(x,y)\, \nu(dy).
\end{equation}

Now, we introduce PCAFs and smooth measures in the strict sense below.

\begin{dfn} [{PCAF, \cite[p.~222]{Fukushima_Oshima_Takeda_11_Dirichlet}}] \label{2. dfn: PCAF}
  An $\filt_{*}$-adapted $[0, \infty]$-valued stochastic process $A = (A_{t})_{t \geq 0}$
  defined on $\Omega$ is called a \emph{positive continuous additive functional (PCAF)} of $X$ if 
  there exist a set $\Lambda \in \filt_{\infty}$, called a \emph{defining set} of $A$, 
  and a properly exceptional set $N \in \mathcal{B}(S)$ for $X$, called an \emph{exceptional set} of $A$,
  satisfying the following.
  \begin{enumerate} [label = \textup{(\roman*)}]
    \item It holds that $P_{x}(\Lambda) = 1$ for all $x \in S \setminus N$ and
      $\theta_{t}(\Lambda) \subseteq \Lambda$ for all $t \in [0, \infty)$.
    \item 
      For every $\omega \in \Lambda$,
      $[0,\infty) \ni t \mapsto A_{t}(\omega)$ is a $[0,\infty]$-valued continuous function
      with $A_{0}(\omega)=0$ such that for any $s,t \in [0,\infty)$,
      $A_{t}(\omega)<\infty$ if $t < \zeta(\omega)$,
      $A_{t}(\omega)=A_{\zeta(\omega)}(\omega)$ if $t \geq \zeta(\omega)$,
      and $A_{t+s}(\omega)=A_{t}(\omega)+A_{s}(\theta_{t}(\omega))$.
  \end{enumerate} 
  We set $A_{\infty} \coloneqq \sup_{t \in [0,\infty) \cap \mathbb{Q}} A_{t}$ for each such $A$.
  We say that two PCAFs $A = (A_{t})_{t \geq 0}$ and $B = (B_{t})_{t \geq 0}$ of $X$ are \emph{equivalent}
  if $\int_{S}P_{x}(A_{t} \not= B_{t}) \, m(dx) = 0$ for all $t \in (0, \infty)$;
  see \cite[Lemma A.3.2]{Chen_Fukushima_12_Symmetric} in this connection.
\end{dfn}

\begin{dfn}[{PCAF in the strict sense, \cite[pp.~235--236]{Fukushima_Oshima_Takeda_11_Dirichlet}}] \label{2. dfn: PCAF in the strict sense}
  A PCAF of $X$ is called a \emph{PCAF in the strict sense} of $X$ if it admits a defining set $\Lambda$ with $P_{x}(\Lambda) = 1$ 
  for all $x \in S$.
  In other words,
  it is a PCAF of $X$ such that the empty set $\emptyset$ can be taken as an exceptional set of it.
  We say that two PCAFs $A = (A_{t})_{t \geq 0}$ and $B = (B_{t})_{t \geq 0}$ in the strict sense of $X$ are \emph{equivalent}
  if $\int_{S}P_{x}(A_{t} \not= B_{t}) \, m(dx) = 0$ for all $t \in (0, \infty)$, which,
  by \cite[Proof of Lemma A.3.2]{Chen_Fukushima_12_Symmetric} combined with
  \ref{2. item: absolute continuity} and \cite[Theorem A.2.17(iii)]{Chen_Fukushima_12_Symmetric},
  holds if and only if there exists a set $\Lambda \in \filt_{\infty}$
  which is a defining set of both $A$ and $B$ and satisfies $P_{x}(\Lambda)=1$ for all $x \in S$
  and $A_{t}(\omega) = B_{t}(\omega)$ for all $(t, \omega) \in [0, \infty) \times \Lambda$.
  We write $\mathbf{A}_{c, 1}^{+}$ for the collection of all PCAFs in the strict sense of $X$.
\end{dfn}

The next definition requires one more piece of notation.
For each $E \subseteq S_{\Delta}$, we define the first exit time $\tau_{E} \colon \Omega \to [0, \infty]$ of $X$ from $E$ by setting
$\tau_{E}(\omega) \coloneqq \inf\{ t \in [0, \infty) \mid X_{t}(\omega) \notin E \}$ ($\inf\emptyset \coloneqq \infty$),
so that $\tau_{E}$ is an $\filt_{*}$-stopping time for any $E \in \mathcal{B}(S_{\Delta})$
by \cite[Theorem A.2.3]{Fukushima_Oshima_Takeda_11_Dirichlet} or \cite[Theorem A.1.19]{Chen_Fukushima_12_Symmetric}.

\begin{dfn} [{Smooth measure in the strict sense, \cite[p.~238]{Fukushima_Oshima_Takeda_11_Dirichlet}}] \label{2. dfn: smooth measures in the strict sense}
  We define $\smooth_{00}$ to be the collection of finite Borel measures $\mu$ on $S$ satisfying $\|\resolop_{1}\mu\|_{\infty} < \infty$.
  A Borel measure $\mu$ on $S$ is called a \emph{smooth measure in the strict sense} 
  if there exists a non-decreasing sequence $(S_{n})_{n \in \mathbb{N}}$ of Borel subsets of $S$
  such that $1_{S_{n}} \cdot \mu \in \smooth_{00}$ for every $n \in \mathbb{N}$ and $P_{x}(\lim_{n \to \infty} \tau_{S_{n}} \geq \zeta) = 1$ for each $x \in S$;
  note that then $S = \bigcup_{n \in \mathbb{N}}S_{n}$ since, for any $x \in S$,
  $P_{x}(\lim_{n \to \infty} \tau_{S_{n}} > 0) \geq P_{x}(\lim_{n \to \infty} \tau_{S_{n}} \geq \zeta) = 1$
  and hence $x \in \bigcup_{n \in \mathbb{N}}S_{n}$.
  We write $\smooth_{1}$ for the collection of smooth measures in the strict sense.
\end{dfn}

It is known that there is a one-to-one correspondence between the set of equivalence classes of
PCAFs of $X$ and a certain class of Borel measures on $S$ called smooth measures.
This correspondence is due to Revuz \cite{Revuz_70_Mesures} and referred to as the \emph{Revuz correspondence},
and the smooth measure $\mu_{A}$ corresponding under it to a given PCAF $A$ of $X$ is given by
\begin{equation} \label{2. eq: Revuz measure}
  \mu_{A}(E) 
  \coloneqq 
  \sup_{t \in (0,\infty)} \frac{1}{t} \int_{S} E_{x}\biggl[ \int_{0}^{t} 1_{E}(X_{s})\, dA_{s} \biggr] \, m(dx)
  =
  \lim_{t\downarrow 0} \frac{1}{t} \int_{S} E_{x}\biggl[ \int_{0}^{t} 1_{E}(X_{s})\, dA_{s} \biggr] \, m(dx)
\end{equation}
for every Borel subset $E$ of $S$ and called the \emph{Revuz measure} of $A$;
for details see also \cite[Subsection A.3.1 and Theorem 4.1.1]{Chen_Fukushima_12_Symmetric} and
\cite[Theorems 5.1.3 and 5.1.4]{Fukushima_Oshima_Takeda_11_Dirichlet}, which we follow for the terminology.
By \cite[Theorem 5.1.7]{Fukushima_Oshima_Takeda_11_Dirichlet}, restricting the Revuz correspondence gives
a one-to-one correspondence between the equivalence classes of PCAFs in the strict sense of $X$ and the
smooth measures in the strict sense, and we also have the following characterization of this correspondence.

\begin{lem} [{Revuz correspondence}] \label{2. lem: Revuz correspondence in the strict sense}
  Let $A = (A_{t})_{t \geq 0} \in \mathbf{A}_{c, 1}^{+}$,
  let $\mu$ be a $\sigma$-finite Borel measure on $S$, and
  consider the following condition \ref{2. item: Revuz correspondence}.
  \begin{enumerate} [label = \textup{(RC)}, leftmargin = *]
    \item \label{2. item: Revuz correspondence}
      For any $\alpha \in (0, \infty)$, $x \in S$, and Borel measurable function $f \colon S \to [0, \infty]$,
      \begin{equation} \label{2. eq: Revuz correspondence}
        E_{x}\biggl[
          \int_{0}^{\infty} e^{-\alpha t} f(X_{t})\, dA_{t}
        \biggr]
        = 
        \int_{S} \resol_{\alpha}(x,y) f(y)\, \mu(dy).
      \end{equation}
  \end{enumerate}
  Then, $\mu$ is the Revuz measure of $A$ if and only if \ref{2. item: Revuz correspondence} holds.
\end{lem}

\begin{proof}
  Let $\mu_{A}$ denote the Revuz measure of $A$.
  Note that by the symmetry of $p_{t}$ and Fubini's theorem we have, for any $\alpha \in (0,\infty)$ and $y \in S$,
  $\alpha \int_{S} \resol_{\alpha}(x,y) \, m(dx) = \int_{0}^{\infty} e^{-s} P_{y}( s/\alpha < \zeta ) \, ds$,
  which is non-decreasing in $\alpha$ and converges to $P_{y}( 0 < \zeta ) = 1$ as $\alpha \to \infty$
  by the monotone convergence theorem. Therefore $\alpha$ times the $m(dx)$-integrals on $S$
  of the left- and right-hand sides of \eqref{2. eq: Revuz correspondence} converge as $\alpha \to \infty$,
  respectively, to $\int_{S} f \, d\mu_{A}$ by \cite[Theorem A.3.5(iv)]{Chen_Fukushima_12_Symmetric}
  and to $\int_{S} f \, d\mu$ by Fubini's and the monotone convergence theorems,
  whence \ref{2. item: Revuz correspondence} implies $\mu = \mu_{A}$.
  Conversely, if $\mu = \mu_{A}$, then \ref{2. item: Revuz correspondence} holds
  by \cite[Th\'{e}or\`{e}me~V.5]{Revuz_70_Mesures} (or alternatively one can verify
  \ref{2. item: Revuz correspondence} by following \cite[Proof of Proposition 2.32]{Kajino_Murugan_pre_Heat}).
\end{proof}

\begin{rem} \label{2. rem: PCAF in the strict sense outside defining set}
  Let $A=(A_{t})_{t \geq 0}$ be a PCAF in the strict sense of $X$ with defining set $\Lambda$.
  Then, noting that $\{ \zeta = 0 \} = \{ X_{0} = \Delta \}$ satisfies either
  $\{ \zeta = 0 \} \subseteq \Omega_{0}$ or $\{ \zeta = 0 \} \cap \Omega_{0} = \emptyset$
  for any $\Omega_{0} \in \filt_{\infty}$, we easily see that $\Lambda \in \filt_{0}$,
  and thus that $( A_{t} 1_{\Lambda} )_{t \geq 0}$ is a PCAF in the strict sense of $X$
  equivalent to $A$ with defining set $\Lambda \cup \{\zeta = 0\} \in \filt_{0}$.
  In view of this observation and the Revuz correspondence described just before
  Lemma \ref{2. lem: Revuz correspondence in the strict sense},
  we may and do assume without loss of generality that \emph{every PCAF $A=(A_{t})_{t \geq 0}$
  in the strict sense of $X$ with defining set $\Lambda$ considered in the rest of this paper
  satisfies $A_{t}(\omega)=0$ for any $(t,\omega)\in[0,\infty)\times(\Omega\setminus\Lambda)$
  and $\{ \zeta=0 \} \subseteq \Lambda$}.
\end{rem}

The following fact will be used in the proof of Proposition \ref{prop: improved kac formula} below. 

\begin{lem} [{\cite[Exercise A.3.3 and Theorem A.3.5(iii)]{Chen_Fukushima_12_Symmetric}}]\label{2. lem: restriction of smoothe measures}
  Let $A=(A_{t})_{t \geq 0}$ be a PCAF in the strict sense of $X$,
  let $\mu$ be the Revuz measure of $A$, and let $E \in \mathcal{B}(S)$.
  Then, the process $B = (B_{t})_{t \geq 0}$ defined by setting $B_{t} \coloneqq \int_{0}^{t} 1_{E}(X_{s})\, dA_{s}$
  is a PCAF in the strict sense of $X$ with Revuz measure $1_{E} \cdot \mu$. 
\end{lem}


\subsection{Main results} \label{sec: main results}

In this subsection, we state our main theorems, Theorems \ref{2. thm: main result} and \ref{2. thm: generalized Kac's formula},
and provide two corollaries of them, Corollaries \ref{2. cor: applications} and \ref{2. cor: finite exponential moment}.
The proofs of Theorems \ref{2. thm: main result} and \ref{2. thm: generalized Kac's formula}
are given later in Section \ref{sec: Proofs of the main results}.
We continue to assume the setting specified in Subsection \ref{sec: PCAFs and smooth measures}.
For measurable spaces $(Y,\mathcal{A})$ and $(Z,\mathcal{B})$, we write
$\mathcal{A} \otimes \mathcal{B}$ for the product $\sigma$-algebra of $\mathcal{A}$ and $\mathcal{B}$ in $Y \times Z$.

\begin{thm} \label{2. thm: main result}
  Let $A=(A_{t})_{t \geq 0}$ be a PCAF in the strict sense of $X$ and
  let $\mu$ be the Revuz measure of $A$.
  Let $F \colon (0,\infty) \times \Omega \to [0,\infty]$ be 
  $\mathcal{B}((0,\infty)) \otimes \filt_{\infty}$-measurable, set
  $F_{t}(\omega) \coloneqq F(t, \omega)$ for $(t, \omega) \in (0,\infty) \times \Omega$,
  and assume that the following conditions are satisfied:
  \begin{enumerate} [label = \textup{(\roman*)}]
    \item \label{2. thm item: main result condition-1} there exists $\Lambda \in \filt_{\infty}$ with $P_{x}(\Lambda) = 1$ for all $x \in S$ such that
      the function $(0,\infty) \times \Omega \ni (t, \omega) \mapsto 1_{\Lambda}(\omega) F_{t} \circ \theta_{t}(\omega) = 1_{\Lambda}(\omega) F_{t}(\theta_{t}\omega)$ is 
      $\mathcal{B}((0,\infty)) \otimes \filt_{\infty}$-measurable;
    \item \label{2. thm item: main result condition-2}
      the function $(0,\infty) \times S \ni (t, x) \mapsto E_{x}[F_{t}]$ is Borel measurable.
  \end{enumerate}
  Then, for any $t \in (0, \infty]$ and $x \in S$,
  \begin{equation} \label{2. eq: main result}
    E_{x}\biggl[ \int_{0}^{t} F_{s} \circ \theta_{s}\, dA_{s} \biggr]
    = 
    \int_{0}^{t} \int_{S} p_{s}(x, y) E_{y}[F_{s}]\, \mu(dy)\, ds.
  \end{equation}
\end{thm}

\begin{rem} \label{2. rem: comment on F-P results}
  As mentioned in Section \ref{sec: Intro}, a formula similar to \eqref{2. eq: main result}
  is stated in \cite[Equation (21)]{Fitzsimmons_Pitman_99_Kac},
  which deals with expectations of Lebesgue--Stieltjes integrals with respect to PCAFs up to Markov killing times of $X$
  rather than up to deterministic times $t$ as in \eqref{2. eq: main result}.
  As stated in Remark \ref{2. rem: extension to part process} below,
  it is possible to extend our results to the case where $t$ is replaced by
  the first exit time $\tau_{D}$ from a non-empty open subset $D$ of $S$,
  one of the most typical examples of Markov killing times of $X$.
  We emphasize here that our proof of Theorem \ref{2. thm: main result} is quite elementary
  and based only on the strong Markov property of $X$ and
  a change-of-variable formula for Lebesgue--Stieltjes integrals,
  whereas it is indicated in \cite{Fitzsimmons_Pitman_99_Kac} that
  the (omitted) proof of \cite[Equation (21)]{Fitzsimmons_Pitman_99_Kac}
  uses the Ray--Knight compactification and is therefore technically demanding.
\end{rem}

By using Theorem \ref{2. thm: main result},
we further obtain the following generalizations of Kac's moment formula \eqref{1. eq: original Kac's formula}
to various random variables involving general PCAFs in the strict sense of $X$.

\begin{thm} \label{2. thm: generalized Kac's formula}
  Let $k \in \mathbb{N}$, and
  for each $i \in \{1,2,\ldots,k\}$, 
  let $A^{(i)} = (A^{(i)}_{t})_{t \geq 0}$ be a PCAF in the strict sense of $X$ and let $\mu_{i}$ be the Revuz measure of $A^{(i)}$.
  Then, for any $t \in (0, \infty]$, $x \in S$, and Borel measurable function $f \colon S_{\Delta} \to [0, \infty]$,
  \begin{align}
    \lefteqn{E_{x}\Biggl[ f(X_{t}) \prod_{i=1}^{k} A^{(i)}_{t} \Biggr]}\\
    &= 
    \sum_{\pi \in \mathfrak{S}_{k}}
    \int_{0}^{t} dt_{1} \int_{t_{1}}^{t} dt_{2} \cdots \int_{t_{k-1}}^{t} dt_{k} 
    \int_{S} \mu_{\pi_{1}}(dy_{1}) \int_{S} \mu_{\pi_{2}}(dy_{2})
    \cdots \int_{S} \mu_{\pi_{k}}(dy_{k}) \int_{S_{\Delta}} m_{\Delta}(dz)\\
    &\qquad 
    p_{t_{1}}(x, y_{1}) p_{t_{2}-t_{1}}(y_{1}, y_{2})
    \cdots p_{t_{k}-t_{k-1}}(y_{k-1}, y_{k}) p_{t-t_{k}}(y_{k}, z) f(z), \label{2. eq: generalized Kac's formula}
  \end{align}
  where $\mathfrak{S}_{k}$ denotes the set of all the bijections $\pi = (\pi_{i})_{i=1}^{k}$ from $\{1,2,\ldots,k\}$ to itself.
\end{thm}

\begin{cor} \label{2. cor: applications}
  Let $A = (A_{t})_{t \geq 0}$ and $B = (B_{t})_{t \geq 0}$ be PCAFs in the strict sense of $X$,
  and let $\mu$ and $\nu$ be the Revuz measures of $A$ and $B$, respectively.
  Then, the following statements hold.
  \begin{enumerate} [label = \textup{(\roman*)}]
    \item  \label{2. cor item: moment formula 1}
      For any $t \in (0, \infty]$, $x \in S$, and Borel measurable function $f \colon S_{\Delta} \to [0, \infty]$,
      \begin{equation} \label{2. cor eq: moment formula 1}
        E_{x}[f(X_{t}) A_{t}]
        = 
        \int_{0}^{t} 
        \int_{S}
        p_{s}(x, y) 
        \int_{S_{\Delta}} p_{t-s}(y, z) f(z)\, m_{\Delta}(dz)\,
        \mu(dy)\, ds.
      \end{equation}
    \item  \label{2. cor item: moment formula 2}
      For any $t \in (0, \infty]$, $x \in S$, and $k \in \mathbb{N}$, 
      \begin{align}
        E_{x}[A_{t}^{k}]
        &= 
        k! 
        \int_{0}^{t} dt_{1} \int_{t_{1}}^{t} dt_{2} \cdots \int_{t_{k-1}}^{t} dt_{k} 
        \int_{S} \mu(dy_{1}) \int_{S} \mu(dy_{2}) \cdots \int_{S} \mu(dy_{k})\\
        &\qquad
        p_{t_{1}}(x, y_{1}) p_{t_{2}-t_{1}}(y_{1}, y_{2}) \cdots p_{t_{k}-t_{k-1}}(y_{k-1}, y_{k}). \label{2. cor eq: moment formula 2}
      \end{align}
    \item  \label{2. cor item: moment formula 3}
      For any $t \in (0, \infty]$ and $x \in S$, 
      \begin{align}
        E_{x}[A_{t} B_{t}]
        &= 
        \int_{0}^{t} dt_{1} \int_{t_{1}}^{t} dt_{2} 
        \int_{S} \mu(dy_{1}) \int_{S} \nu(dy_{2})\,
        p_{t_{1}}(x, y_{1}) p_{t_{2}-t_{1}}(y_{1}, y_{2}) \\
        &\qquad  
        + 
        \int_{0}^{t} dt_{1} \int_{t_{1}}^{t} dt_{2} 
        \int_{S} \nu(dy_{1}) \int_{S} \mu(dy_{2})\,
        p_{t_{1}}(x, y_{1}) p_{t_{2}-t_{1}}(y_{1}, y_{2}). \label{2. cor eq: moment formula 3}
      \end{align}
  \end{enumerate}
\end{cor}

\begin{proof}
  These are immediate consequences of Theorem \ref{2. thm: generalized Kac's formula}.
\end{proof}

\begin{rem} \label{2. rem: extension to part process}
  More generally, we can extend each of \eqref{2. eq: main result}, \eqref{2. eq: generalized Kac's formula},
  \eqref{2. cor eq: moment formula 1}, \eqref{2. cor eq: moment formula 2} and \eqref{2. cor eq: moment formula 3}
  to a similar formula for the random variable given by replacing $t$ with
  $t \wedge \tau_{D}$ (and requiring $f|_{S_{\Delta} \setminus D}$ to be constant)
  for any non-empty open subset $D$ of $S$. Namely, fix any such $D$, let
  $D_{\Delta} = D \cup \{ \Delta_{D} \}$ be the one-point compactification of $D$,
  set $m|_{D} \coloneqq m|_{\mathcal{B}(D)}$, $P_{\Delta_{D}} \coloneqq P_{\Delta}$,
  and for each $t \in [0, \infty]$ define $X^{D}_{t} \colon \Omega \to D_{\Delta}$
  by setting $X^{D}_{t}(\omega) \coloneqq X_{t}(\omega)$ for $\omega \in \{ t < \tau_{D} \}$
  and $X^{D}_{t}(\omega) \coloneqq \Delta_{D}$ for $\omega \in \{ t \geq \tau_{D} \}$.
  Then $X^{D} \coloneqq (\Omega, \filt_{\infty}, (X^{D}_{t})_{t \in [0, \infty]}, (P_{x})_{x \in D_{\Delta}})$,
  called the \emph{part process} of $X$ on $D$ (killed upon exiting $D$),
  is an $m|_{D}$-symmetric Hunt process on $D$ by \cite[Theorem A.2.10 and Lemma 4.1.3]{Fukushima_Oshima_Takeda_11_Dirichlet}
  (see also \cite[Exercise 3.3.7(ii), (3.3.4) and Exercise 4.1.9(i)]{Chen_Fukushima_12_Symmetric}),
  the Dirichlet form of $X^{D}$ is regular by \cite[Theorems 4.4.2 and 4.4.3]{Fukushima_Oshima_Takeda_11_Dirichlet},
  and $X^{D}$ satisfies \ref{2. item: absolute continuity} by \ref{2. item: absolute continuity} of $X$
  and hence has a unique transition density $p^{D}$ with respect to $m|_{D}$. We further set
  $p^{D}_{t}(x, \Delta_{D}) \coloneqq 1 - \int_{D} p^{D}_{t}(x, y)\, m(dy) = P_{x}( t \geq \tau_{D})$
  and $p^{D}_{\infty}(x, y) \coloneqq 1_{\{\Delta_{D}\}}(y)$
  for $t\in (0, \infty)$, $x \in D$, and $y \in D_{\Delta}$,
  and extend $m|_{D}$ to a Borel measure $(m|_{D})_{\Delta}$ on $D_{\Delta}$ by setting
  $(m|_{D})_{\Delta} \coloneqq m|_{D}(\cdot \cap D) + \delta_{D_{\Delta},\Delta_{D}}$.
  Then we have the following.
  \begin{equation} \label{2. eq: extension to part process}
  \begin{minipage}{410pt}\itshape
  Theorems \textup{\ref{2. thm: main result}}, \textup{\ref{2. thm: generalized Kac's formula}} and
  Corollary \textup{\ref{2. cor: applications}}, with $S,S_{\Delta},p,m_{\Delta}$
  replaced by $D,D_{\Delta},p^{D},(m|_{D})_{\Delta}$ and with $t,X$
  in the left-hand sides of \eqref{2. eq: main result}, \eqref{2. eq: generalized Kac's formula},
  \eqref{2. cor eq: moment formula 1}, \eqref{2. cor eq: moment formula 2} and \eqref{2. cor eq: moment formula 3}
  replaced by $t \wedge \tau_{D},X^{D}$, hold.
  \end{minipage}
  \end{equation}
  Indeed, Theorem \ref{2. thm: generalized Kac's formula} and Corollary \ref{2. cor: applications}
  (and Lemma \ref{2. lem: basic Kac formula} and Proposition \ref{prop: improved kac formula} below)
  are shown to extend in this form to general $D$ by applying them to $X^{D}$ on the basis of
  the following fact implied by \cite[Exercise 4.1.9 and Proposition 4.1.10]{Chen_Fukushima_12_Symmetric}.
  \begin{equation} \label{2. eq: PCAF for part process}
  \begin{minipage}{388pt}\itshape
  If $A=\{A_{t}\}_{t\geq 0}$ is a PCAF in the strict sense of $X$ with Revuz measure $\mu$,
  then $\{A_{t\wedge\tau_{D}}\}_{t\geq 0}$ is a PCAF in the strict sense of $X^{D}$
  with Revuz measure $\mu|_{D} \coloneqq \mu|_{\mathcal{B}(D)}$.
  \end{minipage}
  \end{equation}
  (To be precise, $\{A_{t\wedge\tau_{D}}\}_{t\geq 0}$ might not be adapted to the minimum augmented
  admissible filtration $\filt^{D}_{*}=\{\filt^{D}_{t}\}_{t\in[0,\infty]}$ of $X^{D}$ in $\Omega$ since $\filt^{D}_{*}$ involves
  smaller families of sets of probability zero than $\filt_{*}$, but $\{A_{t\wedge\tau_{D}}\}_{t\geq 0}$
  is still \emph{equivalent} to an ($\filt^{D}_{*}$-adapted) PCAF $A^{D}$ in the strict sense of $X^{D}$.
  Indeed, such $A^{D}$ can be obtained by applying to $A$ the limit construction of PCAFs of $X$ presented in
  \cite[Proofs of Chapter IV, Theorem 3.16 and Chapter~V, Theorem~2.1]{Blumenthal_Getoor_68_Markov}
  and restricting the approximating PCAFs to the stochastic interval $[0,\tau_{D}]$
  in taking their limits to recover $A$; in this connection see also \cite[Remark 4.16]{Noda_pre_STOM}.)
  Then the proof of Theorem \ref{2. thm: main result} given in
  Section \ref{sec: Proofs of the main results} below is easily seen to extend to general $D$
  by replacing $t,1_{S}$ with $t \wedge \tau_{D},1_{D}$ in \eqref{3. eq: main result proof}
  and using the above extension to $D$ of Proposition \ref{prop: improved kac formula}.
  (Note that this extension of Theorem \ref{2. thm: main result} is \emph{not} implied by
  a mere application of Theorem \ref{2. thm: main result} to $X^{D}$, since
  $\filt^{D}_{\infty} \subseteq \filt_{\tau_{D}} \subseteq \filt_{\infty}$
  and the latter inclusion is usually strict.)
\end{rem}

Following \cite{Kim_Kuwae_17_Analytic}, we define
$\smooth_{\mathrm{EK}} \coloneqq \{ \mu \in \smooth_{1} \mid \lim_{\alpha \to \infty} \|\resolop_{\alpha}\mu\|_{\infty} < 1 \}$,
and call $\smooth_{\mathrm{EK}}$ the \emph{extended Kato class}.
This class of smooth measures plays an important role in the study of Feynman--Kac semigroups (see, e.g., \cite{Getoor_99_Measure,Ying_97_Dirichlet}).
As a consequence of \eqref{2. cor eq: moment formula 2},
we can prove that any PCAF in the strict sense of $X$ with Revuz measure in $\smooth_{\mathrm{EK}}$ has uniformly bounded exponential moments at any finite time,
as follows.
Essentially the same result was proved in \cite{Getoor_99_Measure,Ying_97_Dirichlet},
but we give a detailed proof here since our setting and class $\smooth_{\mathrm{EK}}$ of smooth measures
are different from those in \cite{Getoor_99_Measure,Ying_97_Dirichlet}.

\begin{cor} \label{2. cor: finite exponential moment}
  Let $A = (A_{t})_{t \geq 0}$ be a PCAF in the strict sense of $X$ whose Revuz measure
  belongs to $\smooth_{\mathrm{EK}}$.
  Then, there exist $c_{1}, c_{2} \in (0,\infty)$ such that for any $t \in [0, \infty)$,
  \begin{equation}
    \sup_{x \in S}E_{x}\bigl[ e^{A_{t}} \bigr] \leq c_{1} e^{c_{2} t}.
  \end{equation}
\end{cor}

\begin{proof}
  Let $\mu$ be the Revuz measure of $A$. By $\mu \in \smooth_{\mathrm{EK}}$,
  we can take $\alpha \in (0, \infty)$ satisfying $\|\resolop_{\alpha}\mu\|_{\infty} < 1$.
  Choose $s_{\alpha} \in (0, \infty)$ so that $c \coloneqq e^{s_{\alpha}} \|\resolop_{\alpha}\mu\|_{\infty} < 1$,
  set $t_{\alpha} \coloneqq s_{\alpha}/\alpha$, and let $x \in S$. Then,
  \begin{equation} \label{2. eq: finite exponential moment proof-1}
    \int_{0}^{t_{\alpha}} \int_{S} p_{t}(x, y)\, \mu(dy)\, dt 
    \leq
    e^{\alpha t_{\alpha}}
    \int_{S} \int_{0}^{t_{\alpha}} e^{-\alpha t} p_{t}(x, y) \, dt \, \mu(dy)
    \leq
    e^{\alpha t_{\alpha}} 
    \|\resolop_{\alpha}\mu\|_{\infty}
    = 
    c < 1,
  \end{equation}
  which together with Corollary \ref{2. cor: applications}\ref{2. cor item: moment formula 2}
  shows that $E_{x}[A_{t_{\alpha}}^{k}] \leq k! c^{k}$ for each $k \in \mathbb{N}$.
  This then yields
  \begin{equation} \label{2. eq: finite exponential moment proof-2}
    E_{x}\bigl[ e^{A_{t_{\alpha}}} \bigr] 
    = 
    \sum_{k \geq 0} \frac{1}{k!} E_{x}[A_{t_{\alpha}}^{k}]
    \leq
    \frac{1}{1-c}
    \eqqcolon
    c_{1} \in [1, \infty).
  \end{equation}
  Now for each $t \in [0, \infty)$,
  taking $n = n_{t} \in \mathbb{N}$ such that $(n-1)t_{\alpha} \leq t < n t_{\alpha}$,
  and noting that $E_{\Delta}\bigl[ e^{A_{t_{\alpha}}} \bigr] = 1 \leq c_{1}$ by Remark \ref{2. rem: PCAF in the strict sense outside defining set},
  we see from the Markov property of $X$ and \eqref{2. eq: finite exponential moment proof-2} that
  \begin{equation}
    E_{x}\bigl[ e^{A_{t}} \bigr] 
    \leq
    E_{x}\bigl[ e^{A_{n t_{\alpha}}} \bigr] 
    =
    E_{x}\Bigl[ e^{A_{(n-1)t_{\alpha}}} E_{X_{(n-1)t_{\alpha}}}\bigl[ e^{A_{t_{\alpha}}} \bigr] \Bigr]
    \leq
    c_{1} E_{x}\bigl[ e^{A_{(n-1)t_{\alpha}}} \bigr]
    \leq
    c_{1}^{n}
    \leq
    c_{1}^{1+t/t_{\alpha}},
  \end{equation}
  which completes the proof.
\end{proof}

\begin{rem}
  It is not always the case that 
  a PCAF $A = (A_{t})_{t \geq 0}$ in the strict sense of $X$ with Revuz measure in $\smooth_{\mathrm{EK}}$ has uniformly bounded exponential moments at the life time $\zeta$, 
  i.e., $\sup_{x \in S}E_{x}\bigl[ e^{A_{\zeta}} \bigr] < \infty$. 
  If this is the case,
  then $(X,A)$ is said to be \emph{gaugeable}.
  The gaugeability is important in the study of Feynman--Kac semigroups and Schr\"{o}dinger operators
  and there has been plenty of research on it.
  An analytic characterization of the gaugeability can be found in \cite{Kim_Kuwae_17_Analytic}, for example.
\end{rem}


\section{Proofs of the main theorems} \label{sec: Proofs of the main results}

In this section, we prove our main theorems, Theorems \ref{2. thm: main result} and \ref{2. thm: generalized Kac's formula}.
We continue to assume the setting specified in Subsection \ref{sec: PCAFs and smooth measures}.
The main ingredient of the proofs is the following fact.

\begin{lem} \label{2. lem: basic Kac formula}
  Let $A=(A_{t})_{t \geq 0}$ be a PCAF in the strict sense of $X$,
  let $\mu$ be the Revuz measure of $A$, and
  let $f \colon S \to [0, \infty]$ be Borel measurable.
  Then, for any $t \in (0, \infty]$ and $x \in S$,
  \begin{equation} \label{2. eq: basic Kac formula}
    E_{x}\biggl[
      \int_{0}^{t} f(X_{s})\, dA_{s}
    \biggr]
    = 
    \int_{0}^{t} \int_{S} p_{s}(x,y) f(y)\, \mu(dy)\, ds.
  \end{equation}
\end{lem}

\begin{proof}
  This result is stated in \cite[Proposition~2.32]{Kajino_Murugan_pre_Heat}. 
  However, to ensure that it is implied solely by \ref{2. item: Revuz correspondence}
  without relying on any additional assumptions, we provide a complete proof here.
  Recalling that $\mu \in \smooth_{1}$ as mentioned after \eqref{2. eq: Revuz measure},
  let $(S_{n})_{n \in \mathbb{N}}$ be as in Definition \ref{2. dfn: smooth measures in the strict sense}.
  Let $x \in S$, $n \in \mathbb{N}$, and consider the Borel measures
  $\eta_{1},\eta_{2}$ on $[0,\infty)$ given by $\eta_{1}(B) \coloneqq E_{x}\bigl[\int_{B}((f\wedge n)1_{S_{n}})(X_{s})\,dA_{s}\bigr]$
  and $\eta_{2}(B) \coloneqq \int_{B} \int_{S_{n}} p_{s}(x,y)(f\wedge n)(y)\, \mu(dy)\, ds$ for $B \in \mathcal{B}([0,\infty))$.
  It then follows from Fubini's theorem, \eqref{2. eq: potential density}, \ref{2. item: Revuz correspondence}
  and $1_{S_{n}} \cdot \mu \in \smooth_{00}$ that the Laplace--Stieltjes transforms of $\eta_{1},\eta_{2}$
  coincide and are finite on $(0,\infty)$. Thus $\eta_{1}=\eta_{2}$, namely \eqref{2. eq: basic Kac formula}
  with $(f\wedge n)1_{S_{n}}$ in place of $f$ holds for any $t \in (0,\infty]$, and letting $n \to \infty$
  yields \eqref{2. eq: basic Kac formula} by the monotone convergence theorem.
\end{proof}

A simple application of the monotone class theorem allows us to extend Lemma \ref{2. lem: basic Kac formula} as follows.

\begin{prop} \label{prop: improved kac formula}
  Let $A=(A_{t})_{t \geq 0}$ be a PCAF in the strict sense of $X$,
  let $\mu$ be the Revuz measure of $A$, and
  let $f \colon (0, \infty) \times S \to [0, \infty]$ be Borel measurable.
  Then, with the convention that $f(t,\Delta) \coloneqq 0$ for any $t \in (0, \infty)$,
  it holds that for any $t \in (0, \infty]$ and $x \in S$,
  \begin{equation} \label{eq: pre Kac formula}
    E_{x}\biggl[
      \int_{0}^{t} f(s, X_{s})\, dA_{s}
    \biggr]
    = 
    \int_{0}^{t} \int_{S} p_{s}(x,y) f(s, y)\, \mu(dy)\, ds.
  \end{equation}
\end{prop}

\begin{proof}
  By the monotone convergence theorem,
  it suffices to show \eqref{eq: pre Kac formula} for $t \in (0,\infty)$.
  Let $t \in (0,\infty)$, $x \in S$, and
  assume that \eqref{eq: pre Kac formula} holds when $\mu \in \smooth_{00}$.
  To see \eqref{eq: pre Kac formula} in the general case,
  recalling that $\mu \in \smooth_{1}$ as mentioned after \eqref{2. eq: Revuz measure},
  let $(S_{n})_{n \in \mathbb{N}}$ be as in Definition \ref{2. dfn: smooth measures in the strict sense},
  so that $(S_{n})_{n \in \mathbb{N}}$ is non-decreasing and $S = \bigcup_{n \in \mathbb{N}}S_{n}$.
  By Lemma \ref{2. lem: restriction of smoothe measures}, for each $n \in \mathbb{N}$,
  the process $A^{(n)} = \bigl(A^{(n)}_{t}\bigr)_{t \geq 0}$ defined by setting $A^{(n)}_{t} \coloneqq \int_{0}^{t} 1_{S_{n}}(X_{s})\, dA_{s}$
  is a PCAF in the strict sense of $X$ with Revuz measure $1_{S_{n}} \cdot \mu_{n}$.
  Now we conclude \eqref{eq: pre Kac formula}
  by using the monotone convergence theorem to let $n \to \infty$ in the equality
  \begin{equation}
    E_{x}\biggl[
      \int_{0}^{t} f(s, X_{s}) 1_{S_{n}}(X_{s}) \, dA_{s}
    \biggr]
    = 
    E_{x}\biggl[
      \int_{0}^{t} f(s, X_{s}) \, dA^{(n)}_{s}
    \biggr]
    = 
    \int_{0}^{t} \int_{S} p_{s}(x,y) f(s, y) 1_{S_{n}}(y)\, \mu(dy)\, ds
  \end{equation}
  implied for any $n \in \mathbb{N}$ by the assumed validity of \eqref{eq: pre Kac formula} for $A^{(n)}$ and $1_{S_{n}} \cdot \mu \in \smooth_{00}$.
  
  It thus remains to prove \eqref{eq: pre Kac formula} when $\mu \in \smooth_{00}$.
  Define $\mathcal{H}$ to be the collection of $[0, \infty]$-valued Borel measurable functions on $(0, \infty) \times S$
  satisfying \eqref{eq: pre Kac formula} for any $t \in (0, \infty)$ and $x \in S$.
  Then $\mathcal{H}$ is closed under multiplication by any element of $[0, \infty]$ and under sum,
  and $1_{(0,\infty) \times S} \in \mathcal{H}$ by Lemma \ref{2. lem: basic Kac formula}.
  Moreover, for any $t \in (0, \infty)$, $x \in S$, and $E \in \mathcal{B}(S)$,
  the same argument as \eqref{2. eq: finite exponential moment proof-1} yields
  $\int_{0}^{t} \int_{S} p_{s}(x,y) 1_{E}(y)\, \mu(dy)\, ds \leq e^{t} \|\resolop_{1}\mu\|_{\infty}$,
  which is finite since $\mu \in \smooth_{00}$.
  It therefore follows from Lemma \ref{2. lem: basic Kac formula} that
  $1_{(s_{1}, s_{2}] \times E} \in \mathcal{H}$ for any $s_{1},s_{2} \in [0, \infty)$
  with $s_{1} < s_{2}$ and any $E \in \mathcal{B}(S)$, and hence from the monotone class theorem
  (see \cite[Chapter 0, Proof of Theorem 2.3]{Blumenthal_Getoor_68_Markov}) that
  $\mathcal{H}$ is equal to the set of all $[0, \infty]$-valued Borel measurable functions on $(0,\infty) \times S$,
  completing the proof.
\end{proof}

Now, we are ready to prove Theorems \ref{2. thm: main result} and \ref{2. thm: generalized Kac's formula}.

\begin{proof} [Proof of Theorem \textup{\ref{2. thm: main result}}]
  Set $F_{\infty}(\omega) \coloneqq 0 \eqqcolon F_{0}(\omega)$ for $\omega \in \Omega$
  and, for each $t \in [0, \infty)$,
  define $\tau_{t} \colon \Omega \to [0, \infty]$ by setting
  $\tau_{t}(\omega) \coloneqq \inf\{s \in [0, \infty) \mid A_{s}(\omega) > t\}$
  ($\inf \emptyset \coloneqq \infty$), so that
  $\tau_{t}$ is an $\filt_{*}$-stopping time by \cite[Proposition A.3.8(i)]{Chen_Fukushima_12_Symmetric}
  and the function $[0, \infty) \ni s \mapsto \tau_{s}(\omega) \in [0, \infty]$
  is non-decreasing and right-continuous for any $\omega \in \Omega$.
  In particular, by \cite[Proof of Lemma A.2.4]{Fukushima_Oshima_Takeda_11_Dirichlet},
  as maps in $(s, \omega) \in [0, \infty) \times \Omega$, the functions
  $\tau_{s}(\omega)$ and $1_{[0,t \wedge \zeta(\omega))}(\tau_{s}(\omega))$
  for $t \in (0,\infty]$ are $\mathcal{B}([0,\infty)) \otimes \filt_{\infty}$-measurable,
  the $S_{\Delta}$-valued map $X_{\tau_{s}}(\omega) \coloneqq X_{\tau_{s}(\omega)}(\omega)$
  is $\mathcal{B}([0,\infty)) \otimes \filt_{\infty}/\mathcal{B}(S_{\Delta})$-measurable,
  and hence the functions $1_{\Lambda}(\omega) F_{\tau_{s}} \circ \theta_{\tau_{s}}(\omega)$
  and $E_{X_{\tau_{s}}(\omega)}[ F_{u} ]|_{u = \tau_{s}(\omega)}$
  are $\mathcal{B}([0,\infty)) \otimes \filt_{\infty}$-measurable,
  where $\Lambda$ is as in condition \ref{2. thm item: main result condition-1} and
  $F_{\tau_{s}} \circ \theta_{\tau_{s}}(\omega) \coloneqq F_{\tau_{s}(\omega)}(\theta_{\tau_{s}(\omega)}\omega)$.
  Letting $t \in (0, \infty]$ and $x \in S$, and noting
  $P_{x}(\Lambda)=1$, the $\mathcal{B}([0,\infty)) \otimes \filt_{\infty}$-measurability
  of the functions in the previous sentence, and that
  $1_{[0, t \wedge \zeta)}(\tau_{s})$ is $\filt_{\tau_{s}}$-measurable
  for each $s \in [0,\infty)$ (see, e.g., \cite[Chapter I, Proposition 6.8]{Blumenthal_Getoor_68_Markov}),
  now we obtain
  \begin{align}
    E_{x}\biggl[ \int_{0}^{t} F_{s} \circ \theta_{s}\, dA_{s} \biggr] 
    &=
    E_{x}\biggl[
      \int_{0}^{\infty}
      1_{[0, t \wedge \zeta)}(\tau_{s}) 1_{\Lambda} \cdot F_{\tau_{s}} \circ \theta_{\tau_{s}}
    \, ds \biggr] \quad \textrm{(by \cite[Chapter V, Lemma 2.2]{Blumenthal_Getoor_68_Markov})} \\
    &= 
    \int_{0}^{\infty} 
      E_{x}\bigl[ 1_{[0, t \wedge \zeta)}(\tau_{s}) 1_{\Lambda} \cdot F_{\tau_{s}} \circ \theta_{\tau_{s}} \bigr] 
    \, ds \quad \textrm{(by Fubini's theorem)} \\
    &=
    \int_{0}^{\infty} 
      E_{x}\bigl[ 1_{[0, t \wedge \zeta)}(\tau_{s})
        E_{x}[ F_{\tau_{s}} \circ \theta_{\tau_{s}} \mid \filt_{\tau_{s}}] 
      \bigr] 
    \, ds \\
    &= 
    \int_{0}^{\infty} 
      E_{x}\bigl[ 1_{[0, t \wedge \zeta)}(\tau_{s})
        E_{X_{\tau_{s}}}[ F_{u} ]|_{u = \tau_{s}} 
      \bigr] 
    \, ds \\ 
    &= 
    E_{x}\biggl[
      \int_{0}^{\infty} 
        1_{[0, t \wedge \zeta)}(\tau_{s})
        E_{X_{\tau_{s}}}[ F_{u} ]|_{u = \tau_{s}}
      \, ds
    \biggr] \quad \textrm{(by Fubini's theorem)} \\
    &=
    E_{x}\biggl[
      \int_{0}^{t} 1_{S}(X_{s}) E_{X_{s}}[F_{s}]\, dA_{s}
    \biggr] \quad \textrm{(by \cite[Chapter V, Lemma 2.2]{Blumenthal_Getoor_68_Markov});}
    \label{3. eq: main result proof}
  \end{align}
  here the fourth equality follows by applying
  \cite[Chapter I, Exercise 8.16]{Blumenthal_Getoor_68_Markov},
  which holds by the strong Markov property of $X$ (see, e.g., \cite[Theorem A.1.21]{Chen_Fukushima_12_Symmetric}),
  to the $\filt_{*}$-stopping time $\tau_{s}$ and
  the $\filt_{\tau_{s}} \otimes \filt_{\infty}$-measurable function
  $\Omega \times \Omega \ni (\omega, \omega') \mapsto F_{\tau_{s}(\omega)}(\omega')$.
  Combining \eqref{3. eq: main result proof} with Proposition \ref{prop: improved kac formula},
  we get \eqref{2. eq: main result}.
\end{proof}

\begin{rem} \label{3. rem: optional projection}
  The arguments in \eqref{3. eq: main result proof} can be interpreted in terms of the general theory
  of optional projections (see \cite[pp.~390--391 and Section~22]{Sharpe_88_General})
  as identifying the optional projection of $(1_{S}(X_{t}) F_{t} \circ \theta_{t})_{t \geq 0}$
  as $(1_{S}(X_{t}) E_{X_{t}}[F_{t}])_{t \geq 0}$ and then
  applying \cite[Equation~(A5.18)]{Sharpe_88_General} (for an arbitrarily chosen
  $\filt_{\infty}$-measurable function $F_{0} \colon \Omega \to [0,\infty]$).
  Note, however, that this observation does not allow us to reduce the proof of
  \eqref{3. eq: main result proof} to a mere application of \cite[Equation~(A5.18)]{Sharpe_88_General},
  because the former identification is a non-trivial consequence of a version
  \cite[Chapter I, Exercise 8.16]{Blumenthal_Getoor_68_Markov} of the strong Markov property of $X$.
\end{rem}

\begin{proof} [Proof of Theorem \textup{\ref{2. thm: generalized Kac's formula}}]
  Let $f \colon S_{\Delta} \to [0, \infty]$ be Borel measurable.
  First, for any $t \in (0,\infty]$ and $x \in S$,
  \begin{equation}
    E_{x}\Biggl[ f(X_{t}) \prod_{i=1}^{k} A^{(i)}_{t} \Biggr]
    = 
    \sum_{\pi \in \mathfrak{S}_{k}}
    E_{x}\biggl[
      \int_{0}^{t} dA^{(\pi_{1})}_{t_{1}} 
      \int_{t_{1}}^{t} dA^{(\pi_{2})}_{t_{2}} 
      \cdots 
      \int_{t_{k-1}}^{t} dA^{(\pi_{k})}_{t_{k}} 
      f(X_{t})
    \biggr].
  \end{equation}
  It thus suffices to show by induction on $k \in \mathbb{N}$ that,
  for any $A^{(1)},\ldots,A^{(k)}$ with their respective Revuz measures $\mu_{1},\ldots,\mu_{k}$
  as in the statement, $t \in (0,\infty]$, and $x \in S$,
  \begin{align} 
    \lefteqn{E_{x}\biggl[
      \int_{0}^{t} dA^{(1)}_{t_{1}} 
      \int_{t_{1}}^{t} dA^{(2)}_{t_{2}} 
      \cdots 
      \int_{t_{k-1}}^{t} dA^{(k)}_{t_{k}} 
      f(X_{t})
    \biggr]}\\
    &=
    \int_{0}^{t} dt_{1} \int_{t_{1}}^{t} dt_{2} \cdots \int_{t_{k-1}}^{t} dt_{k} 
    \int_{S} \mu_{1}(dy_{1}) \int_{S} \mu_{2}(dy_{2})
    \cdots \int_{S} \mu_{k}(dy_{k}) \int_{S_{\Delta}} m_{\Delta}(dz)\\
    &\qquad
    p_{t_{1}}(x, y_{1}) p_{t_{2}-t_{1}}(y_{1}, y_{2})
    \cdots p_{t_{k}-t_{k-1}}(y_{k-1}, y_{k}) p_{t-t_{k}}(y_{k}, z) f(z).
    \label{3. eq: the goal of induction}
  \end{align}
  This claim holds for $k=1$; indeed,
  applying Theorem \ref{2. thm: main result} with $F_{s} = f(X_{(t-s) \vee 0})$, we have
  \begin{align}
    E_{x}\biggl[ \int_{0}^{t} dA^{(1)}_{s} f(X_{t}) \biggr]
    &= 
    E_{x}\biggl[ \int_{0}^{t} f(X_{t-s}) \circ \theta_{s}\, dA^{(1)}_{s} \biggr]\\
    &= 
    \int_{0}^{t} \int_{S} p_{s}(x, y) E_{y}[f(X_{t-s})]\, \mu_{1}(dy)\, ds
    \quad \textrm{(by Theorem \ref{2. thm: main result})} \\
    &= 
    \int_{0}^{t} \int_{S} \int_{S_{\Delta}} p_{s}(x, y) p_{t-s}(y, z) f(z)\, m_{\Delta}(dz)\, \mu_{1}(dy)\, ds
    \quad \textrm{(by \eqref{2. eq: extended heat kernel property}).}
  \end{align}
  
  Next, let $k \geq 2$ and suppose that the above claim with $k-1$ in place of $k$ holds.
  Let $t \in (0, \infty]$ and, for each $s \in (0, \infty)$,
  define $F^{(1)}_{s},F^{(2)}_{s},F_{s} \colon \Omega \to [0, \infty]$ by setting
  \begin{equation}
  F^{(1)}_{s}
    \coloneqq
    \int_{0}^{(t-s) \vee 0} dA^{(2)}_{t_{2}}
    \int_{t_{2}}^{(t-s) \vee 0} dA^{(3)}_{t_{3}}
    \cdots
    \int_{t_{k-1}}^{(t-s) \vee 0} dA^{(k)}_{t_{k}},
  \mspace{12mu}
  F^{(2)}_{s} \coloneqq f(X_{(t-s) \vee 0}),
  \mspace{12mu}
  F_{s} \coloneqq F^{(1)}_{s} F^{(2)}_{s},
  \end{equation}
  so that, taking a defining set $\Lambda_{i} \in \filt_{\infty}$ of $A^{(i)}$ for each $i \in \{2,\ldots,k\}$
  and setting $\Lambda \coloneqq \bigcap_{i=2}^{k}\Lambda_{i}$, we have $P_{x}(\Lambda) = 1$ for all $x \in S$ and
  \begin{equation} \label{3. eq: the latter of induction}
  1_{\Lambda} \cdot F^{(1)}_{s} \circ \theta_{s}
    =
    1_{\Lambda}
    \int_{s \wedge t}^{t} dA^{(2)}_{t_{2}}
    \int_{t_{2}}^{t} dA^{(3)}_{t_{3}}
    \cdots \int_{t_{k-1}}^{t} dA^{(k)}_{t_{k}}.
  \end{equation}
  In particular, $F^{(1)}_{s}(\omega)$ and $1_{\Lambda}(\omega) F^{(1)}_{s} \circ \theta_{s}(\omega)$
  are right-continuous as $[0, \infty]$-valued maps in $s \in (0, \infty)$ for each $\omega \in \Omega$
  by the monotone convergence theorem, and hence are $\mathcal{B}((0,\infty)) \otimes \filt_{\infty}$-measurable
  as functions in $(s,\omega) \in (0, \infty) \times \Omega$ by \cite[Proof of Lemma A.2.4]{Fukushima_Oshima_Takeda_11_Dirichlet}.
  Moreover, for any $s \in (0, \infty)$ and $x \in S$, $E_{x}[ F_{s} ] = 0$ if $s \geq t$,
  otherwise by the induction hypothesis $E_{x}[ F_{s} ]$ is equal to the right-hand side of
  \eqref{3. eq: the goal of induction} with $k,A^{(i)},\mu_{i},t$ replaced by $k-1,A^{(i+1)},\mu_{i+1},t-s$,
  and therefore $(0, \infty) \times S \ni (s,x) \mapsto E_{x}[ F_{s} ]$ is Borel measurable.
  Theorem \ref{2. thm: main result} is thus applicable to $F_{s} = F^{(1)}_{s} F^{(2)}_{s}$
  and, combined with \eqref{3. eq: the latter of induction}, shows that for any $x \in S$,
  \begin{align}
    E_{x}\biggl[
      \int_{0}^{t} dA^{(1)}_{t_{1}} 
      \int_{t_{1}}^{t} dA^{(2)}_{t_{2}} 
      \cdots 
      \int_{t_{k-1}}^{t} dA^{(k)}_{t_{k}} 
      f(X_{t})
    \biggr]
    &= 
    E_{x}\biggl[
      \int_{0}^{t} F_{t_{1}} \circ \theta_{t_{1}} \, dA^{(1)}_{t_{1}} 
    \biggr] \\
    &=
    \int_{0}^{t} \int_{S}
      p_{t_{1}}(x, y_{1}) E_{y_{1}}[ F_{t_{1}} ]
      \, \mu_{1}(dy_{1}) \, dt_{1},
  \end{align}
  which, together with the above-mentioned expression of $E_{y_{1}}[ F_{t_{1}} ]$
  implied by the induction hypothesis, yields \eqref{3. eq: the goal of induction}
  and thereby completes the proof.
\end{proof}

\section{Extension to standard processes with duals} \label{sec: extension to processes with duals}

As already mentioned in the second paragraph of Section \ref{sec: Intro},
our results in Subsection \ref{sec: main results} hold, without any changes in the proofs,
in the more general framework of a standard process with absolutely continuous transition function
and in duality with another such standard process. In this last section, we set out
this framework and give the background necessary for the discussions in
Subsection \ref{sec: main results} and Section \ref{sec: Proofs of the main results}
to extend immediately to this more general setting.

Recall the notation and conventions introduced at the beginning of Section \ref{sec: Setting and main results}.
In particular, $S$ is a fixed locally compact separable metrizable topological space, and
its one-point compactification is denoted by $S_{\Delta} = S \cup \{\Delta\}$.
Let $X = (\Omega, \sigalg, (X_{t})_{t \in [0, \infty)}, (P_{x})_{x \in S_{\Delta}}, (\theta_{t})_{t \in [0,\infty)})$
be a (Borel) standard process on $S$ (see \cite[the paragraph before Theorem A.2.1]{Fukushima_Oshima_Takeda_11_Dirichlet} or \cite[Definition~A.1.23(i)]{Chen_Fukushima_12_Symmetric}),
and write $\filt_{*} = (\filt_{t})_{t  \in [0,\infty]}$ for the minimum augmented
admissible filtration of $X$ in $\Omega$ (see \cite[p.~397]{Chen_Fukushima_12_Symmetric}).
Let $m$ be a $\sigma$-finite Borel measure on $S$. We assume that $X$ satisfies
the absolute continuity condition \ref{2. item: absolute continuity} and, instead of
being $m$-symmetric, the following duality assumption: there exists a standard process
$\hat{X} = (\hat{\Omega}, \hat{\sigalg}, (\hat{X}_{t})_{t \in [0, \infty)}, (\hat{P}_{x})_{x \in S_{\Delta}}, (\hat{\theta}_{t})_{t \in [0,\infty)})$
on $S$ satisfying \ref{2. item: absolute continuity} such that, for any $t \in (0, \infty)$ and Borel measurable functions $f, g \colon S \to [0,\infty]$,
\begin{equation} \label{4. eq: duality}
  \int_{S} E_{x}[f(X_{t})]\, g(x)\, m(dx) 
  = 
  \int_{S} f(x)\, \hat{E}_{x}[g(\hat{X}_{t})]\, m(dx).
\end{equation}
By \eqref{4. eq: duality}, \ref{2. item: absolute continuity} for $X,\hat{X}$, and \cite[Theorem~2]{Yan_88_A_formula},
there exists a unique Borel measurable function $p \colon (0, \infty) \times S \times S \to [0,\infty]$
satisfying, for any $s, t \in (0, \infty)$ and $x, y \in S$, $P_{x}(X_{t} \in dz) = p_{t}(x,z)\, m(dz)$,
$\hat{P}_{y}(\hat{X}_{t} \in dz) = p_{t}(z,y)\, m(dz)$, and $p_{t+s}(x, y) = \int_{S} p_{t}(x, z) p_{s}(z, y)\, m(dz)$,
and $p$ is called the \emph{transition density} (or \emph{heat kernel}) of $(X,\hat{X})$ (with respect to $m$).
Note that $m$ is then an excessive reference measure for $X$ by \eqref{4. eq: duality},
\ref{2. item: absolute continuity} of $X$ and \cite[Theorem A.2.13(ii)]{Chen_Fukushima_12_Symmetric}
(see \cite[p.~418]{Chen_Fukushima_12_Symmetric} and \cite[Chapter~V, Definition~1.1]{Blumenthal_Getoor_68_Markov}
for the definitions of $m$ being excessive and being a reference measure for $X$, respectively).
We apply the definition of $N \in \mathcal{B}(S)$ being properly exceptional for $X$,
\eqref{2. eq: extended heat kernel}, the definition of $m_{\Delta}$, \eqref{2. eq: potential density}, \eqref{2. eq: potential function},
Definitions \ref{2. dfn: PCAF}, \ref{2. dfn: PCAF in the strict sense}, 
the definition of $\tau_{E}$, and Definition \ref{2. dfn: smooth measures in the strict sense}
all as they are in Subsection \ref{sec: PCAFs and smooth measures}, so that
\eqref{2. eq: extended heat kernel property} and \eqref{2. eq: density resolvent equation} hold.

In the present setting, the Revuz correspondence between PCAFs of $X$ and a class of Borel measures on $S$ is formulated as follows
(see \cite[Subsection A.3.1]{Chen_Fukushima_12_Symmetric} and \cite{Revuz_70_Mesures} for details).
For each PCAF $A = (A_{t})_{t \geq 0}$ of $X$, by \cite[Lemma~A.3.4(ii) and Theorem~A.3.5(i)]{Chen_Fukushima_12_Symmetric}
one can define a Borel measure $\mu_{A}$ on $S$ given by \eqref{2. eq: Revuz measure}, and
$\mu_{A}$ charges no set semipolar for $X$ (i.e.,
$\mu_{A}(E) = 0$ for any $E \in \mathcal{B}(S)$ that is semipolar for $X$) by \cite[Theorem~A.3.5(ii)]{Chen_Fukushima_12_Symmetric}
(see \cite[Definition A.2.6]{Chen_Fukushima_12_Symmetric} for the definition of $E \subseteq S$ being semipolar for $X$).
Furthermore by \cite[Th\'{e}or\`{e}me~VI.1]{Revuz_70_Mesures},
the Revuz measure of a PCAF in the strict sense of $X$ is a smooth measure in the strict sense which,
as already mentioned, charges no set semipolar for $X$, and conversely each such measure
is the Revuz measure of a PCAF $A = (A_{t})_{t \geq 0}$ in the strict sense of $X$,
which is unique up to equivalence (recall Definition \ref{2. dfn: PCAF in the strict sense})
by \cite[Chapter~V, Corollary~2.8]{Blumenthal_Getoor_68_Markov}, \eqref{2. eq: Revuz measure},
\cite[Chapter~IV, Proposition~2.4(i)]{Blumenthal_Getoor_68_Markov},
\ref{2. item: absolute continuity} and \cite[Theorem A.2.17(iii)]{Chen_Fukushima_12_Symmetric}.
(To be precise, the existence of a defining set $\Lambda \in \filt_{\infty}$ of such $A$
is not treated in \cite{Revuz_70_Mesures}, but can be verified by examining
\cite[Proofs of Chapter IV, Theorem 3.16 and Chapter~V, Theorem~2.1]{Blumenthal_Getoor_68_Markov} carefully;
in this connection see also \cite[Remark 4.16]{Noda_pre_STOM}.)
Lemma \ref{2. lem: Revuz correspondence in the strict sense} is proved in the same way as before, with the only change being that
$\alpha \int_{S} \resol_{\alpha}(x,y) \, m(dx) = \int_{0}^{\infty} e^{-s} \hat{P}_{y}( s/\alpha < \hat{\zeta} ) \, ds$
instead, where $\hat{\zeta}$ denotes the life time of $\hat{X}$.
Remark \ref{2. rem: PCAF in the strict sense outside defining set} and Lemma \ref{2. lem: restriction of smoothe measures}
also remain applicable as they are.

Now, with the only exception of Remark \ref{2. rem: extension to part process}, whose necessary
changes are described in Remark \ref{4. rem: extension to part process with dual} below,
all the results and arguments in Subsection \ref{sec: main results} and Section \ref{sec: Proofs of the main results}
remain valid even in the present wider framework, and we thus obtain the following theorem.

\begin{thm} \label{4. thm: main result for process with dual}
  Assume the setting introduced so far in this section.
  Then Theorems \textup{\ref{2. thm: main result}} and \textup{\ref{2. thm: generalized Kac's formula}} hold,
  and consequently so do Corollaries \textup{\ref{2. cor: applications}} and \textup{\ref{2. cor: finite exponential moment}}.
\end{thm}


\begin{rem} \label{4. rem: extension to part process with dual}
Assume the setting introduced so far in this section. Let $D$ be a non-empty open subset of $S$,
let $D_{\Delta} = D \cup \{ \Delta_{D} \}$ be the one-point compactification of $D$,
set $m|_{D} \coloneqq m|_{\mathcal{B}(D)}$, let
$X^{D} = (\Omega, \filt_{\infty}, (X^{D}_{t})_{t \in [0, \infty]}, (P_{x})_{x \in D_{\Delta}})$
be the part process of $X$ on $D$ as defined in Remark \ref{2. rem: extension to part process}, and let
$\hat{X}^{D} = (\hat{\Omega}, \hat{\filt}_{\infty}, (\hat{X}^{D}_{t})_{t \in [0, \infty]}, (\hat{P}_{x})_{x \in D_{\Delta}})$
be the part process of $\hat{X}$ on $D$, which is defined from $\hat{X}$ in exactly the same way as $X^{D}$ is from $X$.
Then $X^{D},\hat{X}^{D}$ are standard processes on $D$ by \cite[Proof of Theorem A.2.10]{Fukushima_Oshima_Takeda_11_Dirichlet}
(see also \cite[Exercises 3.3.7(ii) and 4.1.9(i)]{Chen_Fukushima_12_Symmetric})
and satisfy \ref{2. item: absolute continuity} with respect to $m|_{D}$ by \ref{2. item: absolute continuity} of $X,\hat{X}$,
and we easily see from \cite[Proofs of Lemmas 4.1.2 and 4.1.3]{Fukushima_Oshima_Takeda_11_Dirichlet},
with the uses of the $m$-symmetry of $X$ and \cite[Theorem A.2.4]{Fukushima_Oshima_Takeda_11_Dirichlet}
replaced by those of \eqref{4. eq: duality} and \cite[Chapter I, Theorem 10.16]{Blumenthal_Getoor_68_Markov},
respectively, that $X^{D},\hat{X}^{D}$ satisfy \eqref{4. eq: duality} with respect to $m|_{D}$.
Hence there exists a unique transition density $p^{D}$ of $(X^{D},\hat{X}^{D})$ with respect to $m|_{D}$,
and then, extending $p^{D}$ to $(0, \infty] \times D \times D_{\Delta}$ and $m|_{D}$ to a Borel measure
$(m|_{D})_{\Delta}$ on $D_{\Delta}$ as in Remark \ref{2. rem: extension to part process},
we have the following.
\begin{equation} \label{4. eq: extension to part process with dual}
\emph{The statement \eqref{2. eq: extension to part process} holds in the present setting.}
\end{equation}
Indeed, the arguments in Remark \ref{2. rem: extension to part process} after \eqref{2. eq: extension to part process}
remain applicable and show \eqref{4. eq: extension to part process with dual} (and the analogous extensions of
Lemma \ref{2. lem: basic Kac formula} and Proposition \ref{prop: improved kac formula} to general $D$),
with the only change being that \eqref{2. eq: PCAF for part process} is now implied by
\cite[Exercise 4.1.9]{Chen_Fukushima_12_Symmetric} and \cite[Th\'{e}or\`{e}me~V.5]{Revuz_70_Mesures} instead.

To be precise, this deduction of \eqref{2. eq: PCAF for part process} from \cite[Th\'{e}or\`{e}me~V.5]{Revuz_70_Mesures}
requires a justification because, for each $\alpha \in (0,\infty)$, the function $V^{\alpha}_{B}(x,y)$
in \cite[Th\'{e}or\`{e}me~V.5]{Revuz_70_Mesures} with $B \coloneqq S \setminus D$ might not be defined, or coincide with
$\resol^{D}_{\alpha}(x,y) \coloneqq \int_{0}^{\infty} e^{-\alpha t} p^{D}_{t}(x,y) \, dt$, for any $x,y \in D$.
Indeed, we easily see from Fubini's theorem and the strong Markov property of $\hat{X}$
(see, e.g., \cite[Theorem A.1.21]{Chen_Fukushima_12_Symmetric}) that
$V^{\alpha}_{B}(x,y) = \resol^{D}_{\alpha}(x,y)$ for $m$-a.e.\ $x \in D$ for each $y \in D$,
and then for any Borel measurable function $f \colon D \to [0,\infty]$ we have
\begin{equation}
  E_{x}\biggl[ \int_{0}^{\tau_{D}} e^{-\alpha t} f(X_{t})\, dA_{t} \biggr]
  =
  \int_{D} V^{\alpha}_{B}(x,y) f(y)\, \mu(dy)
  =
  \int_{D} \resol^{D}_{\alpha}(x,y) f(y)\, \mu(dy)
\end{equation}
for $m$-a.e.\ $x \in D$ by \cite[Th\'{e}or\`{e}me~V.5]{Revuz_70_Mesures} and Fubini's theorem
and thus for any $x \in D$ by \cite[Chapter~IV, Proposition~2.4(i)]{Blumenthal_Getoor_68_Markov},
\ref{2. item: absolute continuity} and \cite[Theorem A.2.17(iii)]{Chen_Fukushima_12_Symmetric},
proving \eqref{2. eq: PCAF for part process} by virtue of Lemma \ref{2. lem: Revuz correspondence in the strict sense}.
\end{rem}

\appendix

\section*{Acknowledgements}
The authors thank Dr.\ Takahiro Mori and Dr.\ Takumu Ooi for the information on previous research on gaugeability and the extended Kato class,
and thank Prof.\ Mathav Murugan for his suggestion of mentioning the reference \cite{Molchanov_Ostrovskii_69_stable_trace}.
The authors are grateful also to the anonymous referee for valuable comments and suggestions, 
especially those leading to Remark \ref{3. rem: optional projection} on optional projections and to
Section \ref{sec: extension to processes with duals} on the extension to the setting of a standard process with a dual.
Naotaka Kajino was supported in part by JSPS KAKENHI Grant Number JP23K22399.
Ryoichiro Noda was supported by JSPS KAKENHI Grant Number JP24KJ1447.
This work was supported by the Research Institute for Mathematical Sciences, an International Joint Usage/Research Center located in Kyoto University.

\bibliographystyle{amsplain}
\bibliography{ref_Kac_formula}

\providecommand{\bysame}{\leavevmode\hbox to3em{\hrulefill}\thinspace}
\providecommand{\MR}{\relax\ifhmode\unskip\space\fi MR }
\providecommand{\MRhref}[2]{%
  \href{http://www.ams.org/mathscinet-getitem?mr=#1}{#2}
}
\providecommand{\href}[2]{#2}
\begin{thebibliography}{10}

\bibitem{Blumenthal_Getoor_68_Markov}
R.~M. Blumenthal and R.~K. Getoor, \emph{Markov processes and potential
  theory}, Pure and Applied Mathematics, Vol. 29, Academic Press, New
  York-London, 1968. \MR{0264757}

\bibitem{Chen_Fukushima_12_Symmetric}
Z.-Q. Chen and M.~Fukushima, \emph{Symmetric {M}arkov processes, time change,
  and boundary theory}, London Mathematical Society Monographs Series, vol.~35,
  Princeton University Press, Princeton, NJ, 2012. \MR{2849840}

\bibitem{Darling_Kac_57_occupation}
D.~A. Darling and M.~Kac, \emph{On occupation times for {M}arkoff processes},
  Trans. Amer. Math. Soc. \textbf{84} (1957), 444--458. \MR{84222}

\bibitem{Fitzsimmons_Pitman_99_Kac}
P.~J. Fitzsimmons and J.~Pitman, \emph{Kac's moment formula and the
  {F}eynman-{K}ac formula for additive functionals of a {M}arkov process},
  Stochastic Process. Appl. \textbf{79} (1999), no.~1, 117--134. \MR{1670526}

\bibitem{Fukushima_Oshima_Takeda_11_Dirichlet}
M.~Fukushima, Y.~Oshima, and M.~Takeda, \emph{Dirichlet forms and symmetric
  {M}arkov processes}, extended ed., De Gruyter Studies in Mathematics,
  vol.~19, Walter de Gruyter \& Co., Berlin, 2011. \MR{2778606}

\bibitem{Getoor_99_Measure}
R.~K. Getoor, \emph{Measure perturbations of {M}arkovian semigroups}, Potential
  Anal. \textbf{11} (1999), no.~2, 101--133. \MR{1703827}

\bibitem{Kac_51_connection}
M.~Kac, \emph{On some connections between probability theory and differential
  and integral equations}, Proceedings of the {S}econd {B}erkeley {S}ymposium
  on {M}athematical {S}tatistics and {P}robability, 1950, Univ. California
  Press, Berkeley-Los Angeles, Calif., 1951, pp.~189--215. \MR{45333}

\bibitem{Kajino_Murugan_pre_Heat}
N.~Kajino and M.~Murugan, \emph{Heat kernel estimates for boundary traces of
  reflected diffusions on uniform domains}, 2024, Preprint. {A}vailable at
  arXiv:2312.08546.

\bibitem{Kim_Kuwae_17_Analytic}
D.~Kim and K.~Kuwae, \emph{Analytic characterizations of gaugeability for
  generalized {F}eynman-{K}ac functionals}, Trans. Amer. Math. Soc.
  \textbf{369} (2017), no.~7, 4545--4596. \MR{3632543}

\bibitem{Molchanov_Ostrovskii_69_stable_trace}
S.~A. Molchanov and E.~Ostrovskii, \emph{Symmetric stable processes as traces
  of degenerate diffusion processes}, Theor. Probability Appl. \textbf{14}
  (1969), 128--131. \MR{0247668}

\bibitem{Noda_pre_Continuity}
R.~Noda, \emph{Continuity of the {R}evuz correspondence under the absolute
  continuity condition}, 2025, Preprint. {A}vailable at arXiv:2501.10994.

\bibitem{Noda_pre_STOM}
\bysame, \emph{Convergence of space-time occupation measures of stochastic
  processes and its application to collisions}, 2025, Preprint. {A}vailable at
  arXiv:2510.19936v1.

\bibitem{Revuz_70_Mesures}
D.~Revuz, \emph{Mesures associ\'ees aux fonctionnelles additives de {M}arkov.
  {I}}, Trans. Amer. Math. Soc. \textbf{148} (1970), 501--531. \MR{279890}

\bibitem{Sharpe_88_General}
M.~Sharpe, \emph{General theory of {M}arkov processes}, Pure and Applied
  Mathematics, vol. 133, Academic Press, Inc., Boston, MA, 1988. \MR{958914}

\bibitem{Sznitman_98_Brownian}
A.-S. Sznitman, \emph{Brownian motion, obstacles and random media}, Springer
  Monographs in Mathematics, Springer-Verlag, Berlin, 1998. \MR{1717054}

\bibitem{Yan_88_A_formula}
J.-A. Yan, \emph{A formula for densities of transition functions}, S\'eminaire
  de {P}robabilit\'es, {XXII}, Lecture Notes in Math., vol. 1321, Springer,
  Berlin, 1988, pp.~92--100. \MR{960514}

\bibitem{Ying_97_Dirichlet}
J.~Ying, \emph{Dirichlet forms perturbated by additive functionals of extended
  {K}ato class}, Osaka J. Math. \textbf{34} (1997), no.~4, 933--952.
  \MR{1618693}

\end{thebibliography}
\end{document}